\documentclass{amsproc}
\usepackage{amsmath}
\usepackage{amsfonts}

\theoremstyle{plain}

\newtheorem{lemma}{Lemma}

\newtheorem{theorem}{Theorem}
\numberwithin{equation}{section}

\begin{document}
\title[Concentration of order statistics]{Simultaneous concentration of
order statistics}
\author{Daniel Fresen}
\address{Department of Mathematics, University of Missouri}
\email{djfb6b@mail.missouri.edu}
\dedicatory{To my parents John Fresen and Jill Fresen}
\thanks{I am grateful to John Fresen and Jill Fresen for my education and
for many interesting mathematical discussions throughout the years. \ Many
thanks to Joel Zinn as well as my advisors Alexander Koldobsky and Mark
Rudelson for their comments and suggestions.}
\date{January 2011}

\begin{abstract}
Let $\mu $ be a probability measure on $\mathbb{R}$ with cumulative
distribution function $F$, $(x_{i})_{1}^{n}$ a large i.i.d. sample from $\mu 
$, and $F_{n}$ the associated empirical distribution function.\ The
Glivenko-Cantelli theorem states that with probability 1, $F_{n}$ converges
uniformly to $F$. \ In so doing it describes the macroscopic structure of $%
\{x_{i}\}_{1}^{n}$, however it is insensitive to the position of individual
points. \ Indeed any subset of $o(n)$ points can be perturbed at will
without disturbing the convergence.

We provide several refinements of the Glivenko-Cantelli theorem which are
sensitive not only to the global structure of the sample but also to
individual points. \ Our main result provides conditions that guarantee
simultaneous concentration of all order statistics. \ The example of main
interest is the normal distribution.
\end{abstract}

\maketitle

\section{Introduction}

Let $\mu $ be a probability measure on $\mathbb{R}$ with cumulative
distribution function $F$ and let $(x_{i})_{1}^{\infty }$ denote an i.i.d.
sequence of random variables with distribution $\mu $.\ For each $n\in 
\mathbb{N}$ let $F_{n}$ denote the empirical cumulative distribution function%
\[
F_{n}(t)=\frac{1}{n}|\{i\in \mathbb{N}:i\leq n\text{, }x_{i}\leq t\}|
\]%
where $|A|$ denotes the cardinality of a set $A$.\ The Glivenko-Cantelli
theorem (see e.g. \cite{Du}) states that with probability 1,%
\[
\lim_{n\rightarrow \infty }\sup_{t\in \mathbb{R}}|F(t)-F_{n}(t)|=0
\]%
The Dvoretzky-Kiefer-Wolfowitz inequality (\cite{DKW} and \cite{Ma})
provides a quantitative formulation of this and states that for all $n\in 
\mathbb{N}$ and all $\lambda >0$, with probability at least $1-2\exp
(-2\lambda ^{2})$,%
\[
\sup_{t\in \mathbb{R}}\sqrt{n}|F(t)-F_{n}(t)|\leq \lambda
\]%
This titanic theorem would be well deserving of the name '\textit{the
fundamental theorem of statistics}' as it is the theoretical foundation
behind the idea that a large independent sample is representative of the
population. There is, however, a certain crudeness in\ this noble theorem.
Asymptotically, individual points play a negligible role and we learn very
little about the finer structure of the sample $\{x_{i}\}_{1}^{n}$.\ For
instance, it gives us almost no information about either the maximum or the
minimum. We could take any subset of $o(n)$ points and perturb them as we
please without affecting the convergence.

Donsker's theorem (see e.g. \cite{Do}, \cite{KMT} and \cite{MZ}) gives more
insight into the structure of the sample. \ Consider the stochastic process $%
X_{n}$ defined on $\mathbb{R}$ by \ 
\[
X_{n}(t)=\sqrt{n}(F_{n}(t)-F(t))
\]%
Provided that $F$ is strictly increasing and continuous, $X_{n}$ converges
to a re-scaled Brownian bridge (more precisely, $X_{n}\circ F^{-1}$
converges to a Brownian bridge on $[0,1]$). \ However Donsker's theorem is
plagued by a similar insensitivity to the cries of the minority. \ Through
the eyes of Donsker's theorem, we can 'see' subsets as small as $\sqrt{n}$
but are blind to anything smaller such as subsets of size $\log (n)$.

In this paper we provide refined forms of the Glivenko-Cantelli theorem
which, under certain conditions, guarantee tight control over all or most
points in the sample, not only individually but \textit{simultaneously}. \
Super-exponential decay of the distribution provides simultaneous
concentration of \textit{all} order statistics (see theorem 1) while
exponential decay provides simultaneous concentration of \textit{most} order
statistics and slightly weaker control over the rest (see theorems 2 and 3).
\ We provide quantitative bounds for log-concave distributions (see theorem
4).

Our results extend the Gnedenko law of large numbers, which guarantees
concentration of $\max \{x_{i}\}_{1}^{n}$. \ They may be compared to the
results in \cite{Fr} where the Gnedenko law of large numbers is extended to
the multi-dimensional setting, to the paper \cite{GLSW} that provides
estimates of order statistics in terms of Orlicz functions and to the
article \cite{AKT} that concerns optimal matchings of random points
uniformly distributed within the unit square. We refer the reader to \cite%
{GS} and \cite{SW} for an extensive treatment of empirical process theory
and to \cite{BC}, \cite{Da} and \cite{SG} for information on order
statistics. Interesting papers on the Glivenko-Cantelli theorem include \cite%
{Deh}, \cite{Ta}, \cite{Ta2} and \cite{We}.

\begin{theorem}
Let $\mu $ be any probability measure on $\mathbb{R}$ with a continuous
strictly increasing cumulative distribution function $F$ such that for all $%
\varepsilon >0$%
\begin{equation}
\lim_{t\rightarrow \infty }\frac{1-F(t+\varepsilon )}{1-F(t)}%
=\lim_{t\rightarrow -\infty }\frac{F(t)}{F(t+\varepsilon )}=0  \label{h}
\end{equation}%
Then there exists a sequence $(\delta _{n})_{1}^{\infty }$ with $%
\lim_{n\rightarrow \infty }\delta _{n}=0$ such that for all $n\in \mathbb{N}$%
, if $(x_{i})_{1}^{n}$ is an i.i.d. sample from $\mu $ with corresponding
order statistics $(x_{(i)})_{1}^{n}$, then with probability at least $%
1-\delta _{n}$,%
\begin{equation}
\sup_{1\leq i\leq n}|x_{(i)}-x_{(i)}^{\ast }|\leq \delta _{n}  \label{k}
\end{equation}%
where $x_{(i)}^{\ast }=F^{-1}(i/(n+1))$.
\end{theorem}

\begin{theorem}
Let $\mu $ be any probability measure on $\mathbb{R}$ with a continuous
strictly increasing cumulative distribution function $F$ such that for all $%
\varepsilon >0$%
\begin{eqnarray}
\underset{t\rightarrow \infty }{\lim \sup }\frac{1-F(t+\varepsilon )}{1-F(t)}
&<&1  \label{l} \\
\underset{t\rightarrow -\infty }{\lim \sup }\frac{F(t)}{F(t+\varepsilon )}
&<&1  \label{m}
\end{eqnarray}%
\ Let $(\omega _{n})_{1}^{\infty }$ be any sequence in $\mathbb{N}$ with $%
\lim_{n\rightarrow \infty }\omega _{n}=\infty $. \ Then there exists a
sequence $(\delta _{n})_{1}^{\infty }$ with $\lim_{n\rightarrow \infty
}\delta _{n}=0$, such that for all $n\in \mathbb{N}$, if $(x_{i})_{1}^{n}$
is an i.i.d. sample from $\mu $ with corresponding order statistics $%
(x_{(i)})_{1}^{n}$, then with probability at least $1-\delta _{n}$,%
\[
\sup_{\omega _{n}\leq i\leq n-\omega _{n}}|x_{(i)}-x_{(i)}^{\ast }|\leq
\delta _{n}
\]%
where $x_{(i)}^{\ast }=F^{-1}(i/(n+1))$.
\end{theorem}

\begin{theorem}
Let $\mu $ be any probability measure on $\mathbb{R}$ that obeys the
conditions of theorem 2. \ Then there exists $k>0$ such that for all $%
T>10^{6}$ and all $n\in \mathbb{N}$, if $(x_{i})_{1}^{n}$ is an i.i.d.
sample from $\mu $ with corresponding order statistics $(x_{(i)})_{1}^{n}$,
then with probability at least $1-400T^{-1/2}$,%
\[
\sup_{1\leq i\leq n}|x_{(i)}-x_{(i)}^{\ast }|\leq kT
\]
\end{theorem}

Note that in theorem 2 we can take $(\omega _{n})_{1}^{\infty }$ to grow
arbitrarily slowly, for example let $\omega _{n}=\log \log \log n$. \ We
thus have tight control over almost the entire data set with the exception
of a \textit{very} small proportion of points. \ This is substantially
better than the $\sqrt{n}$ 'visibility' of Donsker's theorem.

A probability measure $\mu $ is called $p$-log-concave for some $p\in
(0,\infty )$ if it has a density function of the form $f(x)=c\exp
(-g(x)^{p}) $ where $g$ is non-negative and convex. \ The $1$-log-concave
distributions are simply referred to as log-concave. If $\mu $ is $p$%
-log-concave then it is also $q$-log-concave for all $1\leq q\leq p$.

\begin{theorem}
Let $p>1\,$, $q>0$ and let $\mu $ be a $p$-log-concave probability measure
on $\mathbb{R}$ with a continuous strictly increasing cumulative
distribution function $F$. \ Then there exists $c>0$ such that for any $n\in 
\mathbb{N}$ and any i.i.d. sample $(x_{i})_{1}^{n}$ from $\mu $ with order
statistics $(x_{(i)})_{1}^{n}$, with probability at least $1-c(\log n)^{-q}$,%
\[
\sup_{1\leq i\leq n}|x_{(i)}-x_{(i)}^{\ast }|\leq c\frac{\log \log n}{(\log
n)^{1-1/p}}
\]%
where $x_{(i)}^{\ast }=F^{-1}(i/(n+1))$.
\end{theorem}

The main idea behind the proof of these theorems is to first analyze the
uniform distribution on $[0,1]$. \ We do this using a powerful
representation of the empirical point process via independent random
variables that allows us to use classical results such as the law of large
numbers (in the form of Chebyshev's inequality) and the law of the iterated
logarithm. \ A key step in this analysis is to exploit the inherent
regularity of order statistics which allows for control over all points
based on an inspection of merely $\log n$ carefully chosen points. We then
transform the points under the action of $F^{-1}$ to analyze the general
case. \ We introduce a new class of metrics on $(0,1)$ defined by%
\begin{equation}
\theta _{p}(x,y)=\max \left\{ \frac{\log (x^{-1}y)}{(\log x^{-1})^{1-1/p}},%
\frac{\log ((1-y)^{-1}(1-x))}{(\log (1-y)^{-1})^{1-1/p}}\right\}  \label{q}
\end{equation}%
for $1\leq p<\infty \ $and $0<x\leq y<1$. \ To see that each $\theta _{p}$
is indeed a metric, note that $\theta _{p}(x,y)$ is decreasing in $x$ and
increasing in $y$ throughout the triangular region $\{(x,y)\in
(0,1)^{2}:x<y\}$. We show that $F^{-1}$ is either Lipschitz or uniformly
continuous with respect to these metrics (depending on the assumptions
imposed on $\mu $). \ After this, our main results become straightforward to
prove.

There are endless variations on the main theme of this paper.\ Our intention
is simply to highlight a phenomenon and introduce methods by which to study
it. Note that our results are purely asymptotic in nature and we can (and
do) assume throughout the paper that $n>n_{0}$\bigskip\ for some $n_{0}\in 
\mathbb{N}$.

\section{The uniform distribution}

Let $(\gamma _{i})_{1}^{n}$ denote an i.i.d. sample from the uniform
distribution on $[0,1]$ with corresponding order statistics $(\gamma
_{(i)})_{1}^{n}$ and let $(z_{i})_{1}^{n+1}$ be an i.i.d. sequence of random
variables that follow the standard exponential distribution. \ For $1\leq
i\leq n$ define%
\[
y_{i}=\left( \sum_{j=1}^{i}z_{j}\right) \left( \sum_{j=1}^{n+1}z_{j}\right)
^{-1}
\]%
It is of great interest to us that $(y_{i})_{1}^{n}$ and $(\gamma
_{(i)})_{1}^{n}$ have the same distribution in $\mathbb{R}^{n}$ (see chapter
5 in \cite{De}). This is nothing but an expression of the fact that the
empirical point process locally resembles the Poisson point process. \ Also
of interest is the fact that these random vectors have the same distribution
as the partial sums of a random vector uniformly distributed (with respect
to Lebesgue measure) in the standard simplex $\Delta ^{n}=\{w\in \mathbb{R}%
^{n+1}:w_{i}\geq 0$ $\forall i$, $\sum_{i}w_{i}=1\}$. \ The power of this
representation is that we have an expression for $(\gamma _{(i)})_{1}^{n}$
in terms of independent random variables. \ Note that%
\begin{equation}
y_{i}=\frac{i}{n+1}\left( \frac{1}{i}\sum_{j=1}^{i}z_{j}\right) \left( \frac{%
1}{n+1}\sum_{j=1}^{n+1}z_{j}\right) ^{-1}  \label{d}
\end{equation}%
Both lemma 1 and lemma 3 below can be compared to the results in \cite{We2}.

\begin{lemma}
Let $T>10^{6}$ and $n\in \mathbb{N}$. \ With probability at least $%
1-400T^{-1/2}$ the following inequalities hold simultaneously for all $1\leq
i\leq n$,%
\begin{equation}
T^{-1}\leq \gamma _{(i)}\left( \frac{i}{n+1}\right) ^{-1}\leq T  \label{e}
\end{equation}%
\begin{equation}
T^{-1}\leq (1-\gamma _{(i)})\left( 1-\frac{i}{n+1}\right) ^{-1}\leq T
\label{f}
\end{equation}
\end{lemma}

\begin{proof}
Let $Q=2^{-1}T^{1/2}$ and momentarily fix $1\leq i\leq n+1$. The random
variable $i^{-1}\sum_{j=1}^{i}z_{j}$ has mean $1$ and variance $i^{-1}$. \
Using Chebyshev's inequality, with probability at least $1-i^{-1}Q^{-2}$ we
have%
\[
-Q<1-\frac{1}{i}\sum_{j=1}^{i}z_{j}<Q
\]%
The random variable%
\[
U_{i}=|\{j\in \mathbb{N}:j\leq i,z_{j}\leq 2Q^{-1}\}|
\]%
follows a binomial distribution with $i$ trials and success probability $%
1-\exp (-2Q^{-1})\leq 2Q^{-1}$. \ Using Chebyshev's inequality again, with
probability at least $1-32i^{-1}Q^{-1}$ we have $U_{i}<i/2$, which implies
that $i^{-1}\sum_{j=1}^{i}z_{j}>Q^{-1}$. Hence, with probability at least $%
1-33i^{-1}Q^{-1}$ we have%
\begin{equation}
Q^{-1}<\frac{1}{i}\sum_{j=1}^{i}z_{j}<Q+1  \label{c}
\end{equation}%
Let $M=\left\lfloor \log _{2}(n)\right\rfloor $. \ With probability at least 
$1-33Q^{-1}\sum_{j=0}^{M}2^{-j}-33(n+1)^{-1}Q^{-1}\geq 1-100Q^{-1}$equation (%
\ref{c}) holds simultaneously for $i=1,2,2^{2},2^{3}\ldots 2^{M}$ and for $%
i=n+1$. Hence, by (\ref{d}), with probability at least $1-100Q^{-1}$ we have
that for all such $i$%
\[
\frac{1}{2}Q^{-2}\frac{i}{n+1}\leq y_{i}\leq 2Q^{2}\frac{i}{n+1}
\]%
Since $(y_{i})_{1}^{n}$ is an increasing sequence, control over the values $%
(y_{2^{j}})_{j=1}^{M}$ leads to control over the entire sequence and,
recalling the representation of $(\gamma _{(i)})_{1}^{n}$ in terms of $%
(y_{i})_{1}^{n}$, the bound (\ref{e}) follows for all $1\leq i\leq n$. The
bound (\ref{f}) then follows by symmetry.
\end{proof}

\begin{lemma}
Let $t\in (0,1)$ and $n\in \mathbb{N}$. \ With probability at least $1-2\exp
(-nt^{2}/5)$ the following inequality holds simultaneously for all $1\leq
i\leq n$,%
\begin{equation}
\left\vert \gamma _{(i)}-\frac{i}{n+1}\right\vert \leq t  \label{p}
\end{equation}
\end{lemma}

\begin{proof}
We can assume without loss of generality that $n^{-1}\leq 2t/3$ (otherwise
the probability bound becomes trivial). Note that since our sample is taken
from the uniform distribution we have%
\begin{eqnarray*}
\sup_{1\leq i\leq n}|\gamma _{(i)}-i(n+1)^{-1}| &\leq &n^{-1}+\sup_{1\leq
i\leq n}|\gamma _{(i)}-in^{-1}| \\
&=&n^{-1}+\sup_{0\leq t\leq 1}|F_{n}(t)-F(t)|
\end{eqnarray*}%
where $F(t)=t$ is the cumulative distribution function and $F_{n}$ is the
empirical distribution function. \ By the Dvoretzky-Kiefer-Wolfowitz
inequality (as mentioned in the introduction), with probability at least $%
1-2\exp (-5^{-1}nt^{2})$ we have%
\[
\sup_{0\leq t\leq 1}|F_{n}(t)-F(t)|\leq t/3
\]%
and the result follows.
\end{proof}

Note that in the preceding proof one can also use Doob's martingale
inequality (in the form of Kolmogorov's inequality) and the representation
of $(\gamma _{(i)})_{1}^{n}$ in terms of $(y_{n})_{1}^{n}$, although this
approach yields an inferior probability bound.

\begin{lemma}
Let $(\omega _{n})_{1}^{\infty }$ be any sequence in $\mathbb{N}$ such that $%
\lim_{n\rightarrow \infty }\omega _{n}=\infty $. \ Then for all $T>1$ and
all $\delta \in (0,1)$ there exists $n_{0}\in \mathbb{N}$ such that for all $%
n>n_{0}$, if $(\gamma _{(i)})_{1}^{n}$ are the order statistics from an
i.i.d. sample from the uniform distribution on $[0,1]$, then with
probability at least $1-\delta $, (\ref{e}) and (\ref{f}) hold for all $%
\omega _{n}\leq i\leq n-\omega _{n}$.
\end{lemma}

\begin{proof}
We use the representation (\ref{d}). \ Let $T>1$ and $\delta \in (0,1)$ be
given. \ Without loss of generality we may assume that $T\leq 2$. Let $(%
\widetilde{z}_{i})_{1}^{\infty }$ denote any i.i.d. sequence of random
variables that follow the standard exponential distribution. \ Define the
deterministic sequence $(\lambda _{j})_{1}^{\infty }$ as follows,%
\[
\lambda _{j}=\mathbb{P}\{\sup_{i\geq j}(2i\log \log i)^{-1/2}\left\vert
\sum\limits_{k=1}^{i}(\widetilde{z}_{k}-1)\right\vert \leq 2\}
\]%
Note that $(\lambda _{j})_{1}^{\infty }$ is an increasing sequence and by
the law of the iterated logarithm, $\lim_{j\rightarrow \infty }\lambda
_{j}=1 $. \ Fix $n_{0}\in \mathbb{N}$ with $n_{0}\geq 64\delta
^{-1}(T^{1/2}-1)^{-2} $ such that for all $n>n_{0}$ we have the following
inequalities,%
\begin{eqnarray*}
\lambda _{\omega (n)} &\geq &1-\delta /4 \\
\left( \frac{8\log \log \omega _{n}}{\omega _{n}}\right) ^{1/2} &\leq
&T^{1/2}-1
\end{eqnarray*}%
Now consider any $n>n_{0}$ and let $(\gamma _{(i)})_{1}^{n}$ denote the
order statistics mentioned in the statement of the lemma. \ With probability
at least $1-\delta /4$, for all $\omega (n)\leq i\leq n$,%
\begin{eqnarray*}
\left\vert 1-\frac{1}{i}\sum_{j=1}^{i}z_{j}\right\vert &\leq &\left( \frac{%
8\log \log \omega _{n}}{\omega _{n}}\right) ^{1/2} \\
&\leq &T^{1/2}-1
\end{eqnarray*}%
By Chebyshev's inequality and the fact that the function $u\mapsto u^{-1}$
is 4-Lipschitz on $[1/2,\infty )$, with probability at least $%
1-16n^{-1}(T^{1/2}-1)^{-2}\geq 1-\delta /4$%
\[
\left\vert 1-\left( \frac{1}{n+1}\sum_{j=1}^{n+1}z_{j}\right)
^{-1}\right\vert <T^{1/2}-1
\]%
By (\ref{d}), with probability at least $1-\delta /2$, (\ref{e}) holds for
all $\omega (n)\leq i\leq n$. \ By symmetry, with the same probability (\ref%
{f}) holds for all $1\leq i\leq n-\omega (n)$. \ The lemma is thus proven.
\end{proof}

\bigskip

\section{The general case}

\begin{lemma}
Let $F$ be a continuous strictly increasing cumulative distribution function
that satisfies (\ref{h}). \ Then $F^{-1}$ is continuous and for all $T>1$
and all $\delta >0$ there exists $\eta \in (0,1)$ such that for all $x,y\in
(0,\eta )$ with $T^{-1}\leq xy^{-1}\leq T$ and all $x,y\in (1-\eta ,1)$ with 
$T^{-1}\leq (1-x)(1-y)^{-1}\leq T$ we have $|F^{-1}(x)-F^{-1}(y)|\leq \delta 
$.
\end{lemma}

\begin{proof}
Consider any $T>1$ and $\delta >0$. By (\ref{h}) there exists $t_{0}\in 
\mathbb{R}$ such that for all $t\leq t_{0}$, $TF(t)<F(t+\delta )$. \ Let $%
\eta _{1}=F(t_{0})$. \ Consider any $x,y\in (0,\eta _{1})$ such that $%
T^{-1}\leq xy^{-1}\leq T$. \ Without loss of generality, $x<y$. \ Let $%
s=F^{-1}(x)$ and $t=F^{-1}(y)$. \ Then $s\leq t_{0}$, hence $F(t)=y\leq
Tx=TF(s)<F(s+\delta )$, from which it follows that $t<s+\delta $ and that $%
|F^{-1}(x)-F^{-1}(y)|\leq \delta $. \ Analysis of the right hand tail is
identical and provides us with $\eta _{2}>0$ such that for all $x,y\in
(1-\eta _{2},1)$ with $T^{-1}\leq (1-x)(1-y)^{-1}\leq T$ we have $%
|F^{-1}(x)-F^{-1}(y)|\leq \delta $. \ \ The result follows with $\eta =\min
\{\eta _{1},\eta _{2}\}$.
\end{proof}

\begin{lemma}
Let $F$ be a continuous strictly increasing cumulative distribution function
that satisfies both (\ref{l}) and (\ref{m}). \ Then $F^{-1}$ is continuous
and for all $\delta >0$ there exists $T>1$ such that for all $x,y\in (0,1)$
such that $T^{-1}\leq xy^{-1}\leq T$ and $T^{-1}\leq (1-x)(1-y)^{-1}\leq T$
we have $|F^{-1}(x)-F^{-1}(y)|\leq \delta $. In particular, $F^{-1}$ is
uniformly continuous with respect to the metric $\theta _{1}$ (see (\ref{q}%
)).
\end{lemma}

\begin{proof}
Consider any $\delta >0$. \ By (\ref{m}) there exists $T_{1}>1$ and $%
t_{0}\in \mathbb{R}$ such that for all $t<t_{0}$, $T_{1}F(t)\leq F(t+\delta
) $. \ Let $\eta _{1}=\min \{F(t_{0}),2^{-1}\}$. \ As in the proof of the
previous lemma, it follows that for all $x,y\in (0,\eta _{1})$ with $%
T_{1}^{-1}\leq xy^{-1}\leq T_{1}$ we have $|F^{-1}(x)-F^{-1}(y)|\leq \delta $%
. \ Similarly (using (\ref{l})), there exists $T_{2}>1$ and $\eta _{2}\in
(2^{-1},1)$ such that for all $x,y\in (\eta _{2},1)$ with $T_{2}^{-1}\leq
(1-x)(1-y)^{-1}\leq T_{2}$ we have $|F^{-1}(x)-F^{-1}(y)|\leq \delta $. \ By
continuity of $F^{-1}$ relative to the standard topology on $(0,1)$, and by
compactness of $[2^{-1}\eta _{1},1-2^{-1}\eta _{2}]$ there exists $0<\delta
^{\prime }<10^{-1}\min \{\eta _{1},\eta _{2}\}$ such that for all $x,y\in
\lbrack 2^{-1}\eta _{1},1-2^{-1}\eta _{2}]$ with $|x-y|<\delta ^{\prime }$
we have $|F^{-1}(x)-F^{-1}(y)|\leq \delta $. \ We leave it to the reader to
verify that the result holds with%
\[
T=\min \{T_{1},T_{2},1+\delta ^{\prime }\}
\]
\end{proof}

\begin{proof}[Proof of theorem 1]
We shall construct a function $h$ that takes an arbitrary $\delta \in (0,1)$
and produces an appropriate $n_{0}=h(\delta )\in \mathbb{N}$. \ Then, using
this function we shall define the desired sequence $(\delta
_{n})_{1}^{\infty }$ that is mentioned in the statement of the theorem. \ To
this end, let $\delta \in (0,1)$ be given. \ Define%
\begin{equation}
T=10^{6}\delta ^{-2}  \label{i}
\end{equation}%
By lemma 4 there exists $\eta \in (0,1)$ such that if $x,y\in (0,\eta )$ and 
$T^{-1}\leq xy^{-1}\leq T$, or $x,y\in (1-\eta ,1)$ and $T^{-1}\leq
(1-x)(1-y)^{-1}\leq T$, then $|F^{-1}(x)-F^{-1}(y)|\leq \delta $. \ By
compactness, $F^{-1}$ is uniformly continuous on $[\eta /2,1-\eta /2]$,
which implies the existence of $t\in (0,\eta /2)$ such that if $x,y\in $ $%
[\eta /2,1-\eta /2]$ and $|x-y|\leq t$, then $|F^{-1}(x)-F^{-1}(y)|\leq
\delta $. Define%
\begin{equation}
n_{0}=\left\lceil 5t^{-2}\log (4\delta ^{-1})\right\rceil  \label{j}
\end{equation}%
and consider any $n\geq n_{0}$. Let $(\gamma _{(i)})_{1}^{n}$ denote the
order statistics corresponding to an i.i.d. sample from the uniform
distribution on $[0,1]$. \ Note that we have the representation%
\begin{equation}
x_{(i)}=F^{-1}(\gamma _{(i)})  \label{n}
\end{equation}%
valid for all $1\leq i\leq n$. \ By lemmas 1 and 2, as well as equations (%
\ref{i}) and (\ref{j}), with probability at least $1-\delta $ inequalities (%
\ref{e}), (\ref{f}) and (\ref{p}) hold simultaneously for all $1\leq i\leq n$%
. Suppose that these inequalities do indeed hold and consider any fixed $%
1\leq i\leq n$.\ Since $t\leq \eta /2$, one of the three sets $[0,\eta ]$, $%
[\eta /2,1-\eta /2]$ and $[1-\eta ,1]$ contains both $\gamma _{(i)}$ and $%
i(n+1)^{-1}$, which implies that $|F^{-1}(\gamma
_{(i)})-F^{-1}(i(n+1)^{-1})|\leq \delta $, which is inequality (\ref{k}).

Define the non-decreasing sequence $(\kappa _{n})_{1}^{\infty }$ by $\kappa
_{n}=\max \{h(e^{-i}):1\leq i\leq n\}$ and set%
\[
\delta _{n}=\exp (-\max \{i\in \mathbb{N}:\kappa _{i}\leq n\})
\]%
where we define $\max \emptyset =0$. It is clear that $\lim_{n\rightarrow
\infty }\delta _{n}=0$. \ Consider any fixed $n\in \mathbb{N}$. If $\{i\in 
\mathbb{N}:\kappa _{i}\leq n\}=\emptyset $ then the probability bound is
trivial, otherwise let $j=\max \{i\in \mathbb{N}:\kappa _{i}\leq n\}$. \ The
result follows by the inequality $h(\delta _{n})=h(e^{-j})\leq \kappa
_{j}\leq n$ and by definition of the function $h$.
\end{proof}

\begin{proof}[Proof of theorems 2 and 3]
The proof is very similar to that of theorem 1. \ We use the representation (%
\ref{n}). \ The main difference is that we use lemmas 3 and 5 instead of
lemmas 1 and 4. \ The details are left to the reader.
\end{proof}

\bigskip

\section{Log-concave distributions}

The following two lemmas are modifications of lemmas 6 and 9 in \cite{Fr}.

\begin{lemma}
Let $\mu $ be a log-concave probability measure on $\mathbb{R}$ with a
continuous strictly increasing cumulative distribution function $F$. \ Then
there exists $c>0$ such that for all $0<x<y<1$,%
\begin{equation}
|F^{-1}(y)-F^{-1}(x)|\leq c\max \left\{ \left\vert F^{-1}(y)\right\vert 
\frac{\log (x^{-1}y)}{\log y^{-1}},\left\vert F^{-1}(x)\right\vert \frac{%
\log ((1-x)/(1-y))}{\log (1-x)^{-1}}\right\}  \label{o}
\end{equation}
\end{lemma}

\begin{proof}
By theorem 5.1 in \cite{LV} (see lemma 5 in \cite{Fr} for a proof) $F$ is
log-concave. \ Hence the function $u(t)=-\log F(t)$ is convex (and strictly
decreasing). \ Let $\mathbb{E}\mu $ denote the centroid of $\mu $ (the
expected value of a random variable with distribution $\mu $). By lemma 5.12
in \cite{LV} (see also lemma 3.3 in \cite{Bo}) $F(\mathbb{E}\mu )\geq e^{-1}$%
, hence $u(\mathbb{E}\mu )\leq 1$. \ By convexity of $u$ we have the
inequality $(t-s)^{-1}(u(t)-u(s))\leq (\mathbb{E}\mu -t)^{-1}(u(\mathbb{E}%
\mu )-u(t))$, which is valid for all $s<t<\mathbb{E}\mu $. \ Let $0<x<y<\min
\{e^{-2},F(0),F(-2\mathbb{E}\mu )\}$ and define $s=F^{-1}(x)$ and $%
t=F^{-1}(y)$. \ Then we have%
\[
F^{-1}(y)-F^{-1}(x)\leq (\mathbb{E}\mu -F^{-1}(y))\frac{\log (x^{-1}y)}{\log
y^{-1}-u(\mathbb{E}\mu )}
\]
It follows from the restrictions on $y$ that $F^{-1}(y)<0$ and that $%
\left\vert F^{-1}(y)\right\vert \geq 2\left\vert \mathbb{E}\mu \right\vert $%
. \ Since $y<F(\mathbb{E}\mu )^{2}$, it follows that $\log y^{-1}>2u(\mathbb{%
E}\mu )$ and (\ref{o}) follows for such $x$ and $y$ with $c=4$. \ For other
values of $x$ and $y$, inequality (\ref{o}) follows by compactness,
continuity and symmetry.
\end{proof}

\begin{lemma}
Let $p\geq 1$ and let $\mu $ be a $p$-log-concave probability measure on $%
\mathbb{R}$ with cumulative distribution function $F$. \ Then there exists $%
c>0$ such that for all $x\in (0,1)$,%
\begin{equation}
|F^{-1}(x)|\leq c\max \{(\log x^{-1})^{1/p},(\log (1-x)^{-1})^{1/p}\}
\label{r}
\end{equation}%
As a consequence of (\ref{r}) and (\ref{o}), $F^{-1}$ is Lipschitz with
respect to the metric $\theta _{p}$ (see (\ref{q})).
\end{lemma}

\begin{proof}
By lemma 9 in \cite{Fr} (which holds for $p\geq 1$) there exists $%
c_{1},c_{2}>0$ and $t_{0}>1$ such that for all $t<-t_{0}$, $F(t)\leq
c_{1}|t|^{1-p}\exp (-c_{2}|t|^{p})$. \ Let $\eta _{1}=\min
\{F(-t_{0}),c_{1}^{-1}\}$ and consider any $x\in (0,\eta _{1})$. \ Let $%
t=F^{-1}(x)$. \ Hence $x=F(t)\leq c_{1}|t|^{1-p}\exp (-c_{2}|t|^{p})$, which
implies that%
\begin{eqnarray*}
|F^{-1}(x)| &=&-t \\
&\leq &(c_{2}^{-1}(\log c_{1}+\log x^{-1}))^{1/p} \\
&\leq &2^{1/p}c_{2}^{-1/p}(\log x^{-1})^{1/p}
\end{eqnarray*}%
The result now follows by symmetry, compactness and continuity.
\end{proof}

\begin{lemma}
Let $F$ be a continuous strictly increasing cumulative distribution function
associated to a log-concave probability measure. \ Then there exists $c>0$
such that for all $\varepsilon \in (0,1/2)$ and all $x,y\in \lbrack
\varepsilon ,1-\varepsilon ]$,%
\[
|F^{-1}(x)-F^{-1}(y)|\leq c\varepsilon ^{-1}|x-y|
\]
\end{lemma}

\begin{proof}
This follows from lemmas 6 and 7 with $p=1$ and the inequality $\log t\leq
t-1$.
\end{proof}

\begin{proof}[Proof of theorem 4]
By lemmas 1, 6 and 7, with probability at least $1-400(\log n)^{-q}$, for
all $i\leq n^{3/4}$ and all $i\geq n-n^{3/4}$ we have 
\[
|x_{(i)}-x_{(i)}^{\ast }|\leq c\frac{\log \log n}{(\log n)^{1-1/p}}
\]

Let $I=[2^{-1}n^{-1/4},1-2^{-1}n^{-1/4}]$. \ By lemma 8, for all $x,y\in I$
we have%
\[
|F^{-1}(x)-F^{-1}(y)|\leq cn^{1/4}|x-y|
\]

By lemma 2, with probability at least $1-2\exp (-5n^{1/4})$, for all $1\leq
i\leq n$ we have%
\[
|\gamma _{(i)}-i(n+1)^{-1}|\leq n^{-3/8}
\]

Hence for all $n^{3/4}\leq i\leq n-n^{3/4}$ both $\gamma _{(i)}$ and $%
i(n+1)^{-1}$ are elements of $I$ and the result follows.
\end{proof}

\end{document}